\RequirePackage{fix-cm}
\documentclass{svjour3}                     % onecolumn (standard format)
\smartqed  % flush right qed marks, e.g. at end of proof
\usepackage{graphicx}
 \usepackage{mathptmx}      % use Times fonts if available on your TeX system
\usepackage{latexsym}
\usepackage{amssymb}
\usepackage{amsmath}
\begin{document}

\title{Closed forms for powers and inverses of special matrices%\thanks{Grants or other notes
}


\author{Miloud Mihoubi}


\institute{M. Mihoubi \at
              USTHB, Faculty of Mathematics, P.O. Box 32, El Alia, 16111,
Algiers, Algeria \\
\email{mmihoubi@usthb.dz or miloudmihoubi@gmail.com}
}


\maketitle

\begin{abstract}
This contribution is motivated by old and recent
works on matrix powers and their applications on combinatorial sequences. We
give in this paper the $s$-th powers and the inverses for special upper
triangular matrices and the $s$-th powers for special non-triangular matrices.
The used tools are Lagrange inversion formula and the partial Bell polynomials.
%{MSC code1 \and MSC code2 \and more}
\end{abstract}

\section{Introduction}

The calculation of the matrix power occurs in different mathematical
frameworks such as those occurring in combinatorial sequences, linear
differential equations and statistics. The theory of matrices attracted the
interest of many authors of different frameworks in mathematics, specially in
combinatorics. They gave methods and some interesting algorithms to calculate
their powers, see \cite{aga,apo,hua,kir,leo,liz,shur}. In
combinatorics, to study such combinatorial sequences, some authors refer to
use the matrix representations of sequences, such the works of Aceto and
Ca\c{c}\~{a}o \cite{ace} on a matrix approach to Sheffer polynomials, Ben
Taher and Rachidi \cite{ben} on the matrix powers and exponential by the
$r$-generalized Fibonacci sequences, Chen and Louck \cite{che} on the
combinatorial power of the companion matrix, Rahmani \cite{rah} on the
Akiyama-Tanigawa matrix and related combinatorial identities, Yang and Micek
\cite{yan} on generalized Pascal functional matrix and its applications, and,
Spivey and Zimmer \cite{spi} on symmetric polynomials, Pascal matrices and
Stirling matrices. Motivated by these works, we give in this paper the $s$-th
powers for special upper triangular matrices and for their inverses. We also
give the $s$-th powers for other non-triangular matrices.

\noindent Before starting, we denote for the next by $\left[  \mathcal{A}%
\right]  _{k,n}$ for $\left(  k,n\right)  $-th entry of the matrix
$\mathcal{A},$
\[
\delta_{\left(  a\geq b\right)  }=\left \{
\begin{array}
[c]{l}%
1\text{ if }a\geq b\\
0\text{ otherwise}%
\end{array}
\right.  ,\text{ \  \  \  \ }\delta_{\left(  a\right)  }=\left \{
\begin{array}
[c]{l}%
1\text{ if }a=0\\
0\text{ otherwise,}%
\end{array}
\right.
\]
and, if $\mathcal{A}$ is invertible, the $\left(  -s\right)  $-th power of
$\mathcal{A}$ means the $s$-th power of its inverse $\mathcal{A}^{-1}:$%
\[
\mathcal{A}^{0}:=\mathcal{I}\text{ \ and \ }\mathcal{A}^{-s}:=\left(
\mathcal{A}^{-1}\right)  ^{s},\  \  \  \ s=0,1,2,\ldots,
\]
where $\mathcal{I}$ denotes the identity matrix. Also, for any power series
$\varphi,$ with $\varphi \left(  0\right)  =0\ $and $\varphi^{\prime}\left(
0\right)  \neq0,$ we define the $s$-th composition $\varphi^{\left \langle
s\right \rangle }$ of $\varphi \ $by%
\begin{align*}
\varphi^{\left \langle 0\right \rangle }\left(  t\right)   &  =t,\\
\varphi^{\left \langle 1\right \rangle }\left(  t\right)   &  =\varphi \left(
t\right)  ,\\
\varphi^{\left \langle s\right \rangle }\left(  t\right)   &  =\varphi
^{\left \langle s-1\right \rangle }\circ \varphi \left(  t\right)  =\varphi
^{\left \langle s-1\right \rangle }\left(  \varphi \left(  t\right)  \right)
,\  \  \ s=1,2,3,\ldots \\
&  \text{and}\\
\varphi^{\left \langle -s\right \rangle }\left(  t\right)   &  =\left(
\varphi^{\left \langle -1\right \rangle }\right)  ^{\left \langle s\right \rangle
}\left(  t\right)  ,\text{ \  \ }s=0,1,2,\ldots,
\end{align*}
where $\varphi^{\left \langle -1\right \rangle }$ denotes the compositional
inverse of $\varphi,$ i.e.
\[
\varphi^{\left \langle -1\right \rangle }\circ \varphi \left(  t\right)
=\varphi \circ \varphi^{\left \langle -1\right \rangle }\left(  t\right)  =t.
\]
Moreover, since this contribution is related to partial Bell polynomials, let
us recall that these polynomials are introduced by Bell \cite{bell} and
present a mathematical tool often used the determine the $n$-th derivative of
composite function. They are defined by their generating function%
\begin{equation}
\underset{n\geq k}{\sum}B_{n,k}\left(  \varphi \right)  \dfrac{t^{n}}{n!}%
=\frac{1}{k!}\left(  \varphi \left(  t\right)  \right)  ^{k}. \label{13}%
\end{equation}
For more details on these numbers, one can see \cite{bell,com,mih0,mih3}%
.\newline The principal main results of this paper are given by Theorems
\ref{T1}, \ref{T2} and \ref{T3}.

\section{Powers and inverses of special triangular matrices}

Let $\left(  a_{n};n\geq1\right)  ,$ be a sequence of real numbers with
$a_{j}\neq0,$ $j\geq1,$ $\varphi,g,h$ be power series such that $\varphi
\left(  0\right)  =0,$ $\varphi^{\prime}\left(  0\right)  g\left(  0\right)
h\left(  0\right)  \neq0,$ and let $\mathcal{T}$ be the matrix whose its
$\left(  k,n\right)  $-th entry is%
\begin{equation}
\left[  \mathcal{T}\right]  _{k,n}=\frac{a_{k}}{a_{n}}\frac{1}{k!}\left(
\frac{d}{dt}\right)  _{\! \!t=0}^{\! \!n-1}\left(  \left(  h\left(  t\right)
\right)  ^{n}\frac{d}{dt}\left(  \left(  \varphi \left(  t\right)  \right)
^{k}g\left(  t\right)  \right)  \right)  ,\  \ n\geq k\geq1. \label{1}%
\end{equation}
We see below that $\left[  \mathcal{T}\right]  _{k,n}$ must be%
\begin{equation}
\left[  \mathcal{T}\right]  _{k,n}=\frac{a_{k}}{a_{n}}\frac{1}{k!}\left(
\frac{d}{dt}\right)  _{\! \!t=0}^{\! \!n}\left(  \left(  \left(  \varphi
\circ \omega \right)  ^{\left \langle s\right \rangle }\left(  t\right)  \right)
^{k}g\circ \omega \left(  t\right)  \right)  , \label{3}%
\end{equation}
where $z=\omega \left(  t\right)  $ solution of the equation $z=th\left(
z\right)  .$\newline In this section, we give the $s$-th power and the inverse
of an upper triangular matrix whose $\left(  k,n\right)  $-th entry is in form
of (\ref{1}).

\begin{lemma}
\label{L0}Let $\left(  a_{n};n\geq1\right)  $ be a sequence of complex numbers
with $a_{j}\neq0,$ $j\geq1,$ and let $\left(  \mathcal{A}_{i}\right)  $ and
$\left(  \mathcal{B}_{i}\right)  $ be two sequences of matrices such that
\begin{equation}
\left[  \mathcal{A}_{i}\right]  _{k,n}=\frac{a_{k}}{a_{n}}\left[
\mathcal{B}_{i}\right]  _{k,n},\ i\geq1. \label{21}%
\end{equation}
Then%
\begin{equation}
\left[  \mathcal{A}_{1}\cdots \mathcal{A}_{s}\right]  _{k,n}=\frac{a_{k}}%
{a_{n}}\left[  \mathcal{B}_{1}\cdots \mathcal{B}_{s}\right]  _{k,n}%
,\  \  \ s\geq1. \label{22}%
\end{equation}

\end{lemma}

\begin{proof}
There is no proof for $s=1.$ For $s\geq2,$ by definition, we have
\begin{align*}
\left[  \mathcal{A}_{1}\cdots \mathcal{A}_{s}\right]  _{k,n}  &  =\underset
{j_{1}\geq1,\ldots,j_{s-1}\geq1}{\sum}\left[  \mathcal{A}_{1}\right]
_{k,j_{1}}\left[  \mathcal{A}_{2}\right]  _{j_{1},j_{2}}\cdots \left[
\mathcal{A}_{s}\right]  _{j_{s-1},n}\\
&  =\underset{j_{1}\geq1,\ldots,j_{s-1}\geq1}{\sum}\frac{a_{k}}{a_{j_{1}}%
}\left[  \mathcal{B}_{1}\right]  _{k,j_{1}}\frac{a_{j_{1}}}{a_{j_{2}}}\left[
\mathcal{B}_{2}\right]  _{j_{1},j_{2}}\cdots \frac{a_{j_{s}}}{a_{n}}\left[
\mathcal{B}_{s}\right]  _{j_{s-1},n}\\
&  =\frac{a_{k}}{a_{n}}\underset{j_{1}\geq1,\ldots,j_{s-1}\geq1}{\sum}\left[
\mathcal{B}_{1}\right]  _{k,j_{1}}\left[  \mathcal{B}_{2}\right]
_{j_{1},j_{2}}\cdots \left[  \mathcal{B}_{s}\right]  _{j_{s-1},n}\\
&  =\frac{a_{k}}{a_{n}}\left[  \mathcal{B}_{1}\cdots \mathcal{B}_{s}\right]
_{k,n}.
\end{align*}

\end{proof}

\noindent The first main result is given by the following theorem.

\begin{theorem}
[Closed form for $\mathcal{T}^{s},\ s\in \mathbb{N}$]\label{T1}Let
$\mathcal{T}$ be the matrix whose its $\left(  k,n\right)  $-th entry is given
by (\ref{1}). Then, for any non-negative integer $s,$ there holds%
\begin{equation}
\left[  \mathcal{T}^{s}\right]  _{k,n}=\frac{a_{k}}{a_{n}}\frac{1}{k!}\left(
\frac{d}{dt}\right)  _{\! \!t=0}^{\! \!n}\left(  \left(  \left(  \varphi
\circ \omega \right)  ^{\left \langle s\right \rangle }\left(  t\right)  \right)
^{k}\underset{i=0}{\overset{s-1}{\prod}}g\circ \omega \circ \left(  \varphi
\circ \omega \right)  ^{\left \langle i\right \rangle }\left(  t\right)  \right)
, \label{2}%
\end{equation}
where, for $s=0,$ the empty product is evaluated at one.
\end{theorem}

\begin{proof}
By setting $\mathcal{T}^{0}=\mathcal{I},$ it is obvious that the theorem is
true for $s=0.$ For $s\geq1,$ by Lemma \ref{L0}, it suffices to prove the
proposition for $a_{j}=1,$ $j\geq1.$ Since $\mathcal{T}$ is an upper
triangular matrix, then%
\[
\left[  \mathcal{T}^{s+1}\right]  _{k,n}=\left[  \mathcal{T}^{s}%
\mathcal{T}\right]  _{k,n}=\underset{j=k}{\overset{n}{\sum}}\left[
\mathcal{T}^{s}\right]  _{k,j}\left[  \mathcal{T}\right]  _{j,n}.
\]
So, if we set $F_{k,0}\left(  t\right)  =\frac{t^{k}}{k!},\  \ F_{k,s}\left(
t\right)  =\underset{n\geq1}{\sum}\left[  \mathcal{T}^{s}\right]  _{k,n}%
\frac{t^{n}}{n!},\  \ s\geq1,\  \ k\geq1,$ we get%
\begin{align*}
F_{k,s}\left(  t\right)   &  =\underset{n\geq k}{\sum}\left(  \underset
{j=k}{\overset{n}{\sum}}\left[  \mathcal{T}^{s-1}\right]  _{k,j}\left[
\mathcal{T}\right]  _{j,n}\right)  \frac{t^{n}}{n!}\\
&  =\underset{j\geq k}{\sum}\left[  \mathcal{T}^{s-1}\right]  _{k,j}\left(
\underset{n\geq j}{\sum}\left[  \mathcal{T}\right]  _{j,n}\frac{t^{n}}%
{n!}\right)  \\
&  =\underset{j\geq k}{\sum}\left[  \mathcal{T}^{s-1}\right]  _{k,j}\left(
\frac{1}{j!}\underset{n\geq j}{\sum}\left(  \frac{d}{dt}\right)
_{\! \!t=0}^{\! \!n-1}\left[  \left(  h\left(  t\right)  \right)  ^{n}\frac
{d}{dt}\left(  \left(  \varphi \left(  t\right)  \right)  ^{j}g\left(
t\right)  \right)  \right]  \frac{t^{n}}{n!}\right)
\end{align*}
and by Lagrange inversion formula, the last expansion becomes%
\begin{align*}
F_{k,s}\left(  t\right)   &  =\underset{j\geq k}{\sum}\left[  \mathcal{T}%
^{s-1}\right]  _{k,j}\frac{1}{j!}\left(  \varphi \left(  z\right)  \right)
^{j}g\left(  z\right)  \\
&  =g\left(  z\right)  \underset{j\geq k}{\sum}\left[  \mathcal{T}%
^{s-1}\right]  _{k,j}\frac{\left(  \varphi \left(  z\right)  \right)  ^{j}}%
{j!}\\
&  =g\left(  z\right)  F_{k,s-1}\left(  \varphi \left(  z\right)  \right)  ,
\end{align*}
where $z=\omega \left(  t\right)  :z=th\left(  z\right)  .$ So, recursively on
$s$ we obtain%
\begin{align*}
F_{k,s}\left(  t\right)   &  =g\left(  \omega \left(  t\right)  \right)
F_{k,s-1}\left(  \varphi \circ \omega \left(  t\right)  \right)  \\
&  =g\left(  \omega \left(  t\right)  \right)  g\left(  \omega \circ \left(
\varphi \circ \omega \right)  \left(  t\right)  \right)  F_{k,s-2}\left(  \left(
\varphi \circ \omega \right)  ^{\left \langle 2\right \rangle }\left(  t\right)
\right)  \\
&  =\cdots \\
&  =g\left(  \omega \left(  t\right)  \right)  \cdots g\left(  \omega
\circ \left(  \varphi \circ \omega \right)  ^{\left \langle s-1\right \rangle
}\left(  t\right)  \right)  F_{k,0}\left(  \left(  \varphi \circ \omega \right)
^{\left \langle s\right \rangle }\left(  t\right)  \right)  \\
&  =\frac{1}{k!}\left(  \left(  \varphi \circ \omega \right)  ^{\left \langle
s\right \rangle }\left(  t\right)  \right)  ^{k}g\left(  \omega \left(
t\right)  \right)  \cdots g\left(  \omega \circ \left(  \varphi \circ
\omega \right)  ^{\left \langle s-1\right \rangle }\left(  t\right)  \right)  .
\end{align*}

\end{proof}

\noindent The second main result is given by the following theorem.

\begin{theorem}
[Closed form for $\mathcal{T}^{-s},\ s\in \mathbb{N}$]\label{T2}Let
$\mathcal{T}$ be the matrix whose its $\left(  k,n\right)  $-th entry is given
by (\ref{1}). Then, for any non-negative integer $s,$ there holds%
\begin{equation}
\left[  \mathcal{T}^{-s}\right]  _{k,n}=\frac{a_{k}}{a_{n}}\frac{1}{k!}\left(
\frac{d}{dt}\right)  _{\! \!t=0}^{\! \!n}\left(  \left(  \left(  \frac{\psi
}{h\circ \psi}\right)  ^{\left \langle s\right \rangle }\left(  t\right)
\right)  ^{k}\left(  \underset{i=0}{\overset{s-1}{\prod}}g\circ \psi
\circ \left(  \frac{\psi}{h\circ \psi}\right)  ^{\left \langle i\right \rangle
}\left(  t\right)  \right)  ^{-1}\right)  . \label{5}%
\end{equation}
Here, we have $\frac{\psi}{h\circ \psi}=\left(  \varphi \circ \omega \right)
^{\left \langle -1\right \rangle },$ where $\psi:=\varphi^{\left \langle
-1\right \rangle }$ and for $s=0,$ the empty product is evaluated at one.
\end{theorem}

\begin{proof}
We prove that%
\[
\left[  \mathcal{T}^{-1}\right]  _{k,n}=\frac{a_{k}}{a_{n}}\frac{1}{k!}\left(
\frac{d}{dt}\right)  _{\! \!t=0}^{\! \!n}\left(  \left(  \frac{\psi}%
{h\circ \psi}\left(  t\right)  \right)  ^{k}\left(  g\circ \psi \left(  t\right)
\right)  ^{-1}\right)  .
\]
By Lemma \ref{L0}, it suffices to prove the proposition for $a_{j}=1,$
$j\geq1.$\newline By exponential generating function, we have%
\[
\underset{n\geq k}{\sum}\left[  \mathcal{T}\right]  _{k,n}\frac{t^{n}}%
{n!}=\frac{1}{k!}\left(  \varphi \circ \omega \left(  t\right)  \right)
^{k}g\circ \omega \left(  t\right)
\]
and we must have%
\[
\underset{n\geq k}{\sum}\left[  \mathcal{T}^{-1}\right]  _{k,n}\frac{t^{n}%
}{n!}=\frac{1}{k!}\left(  \frac{\psi}{h\circ \psi}\left(  t\right)  \right)
^{k}\left(  g\circ \psi \left(  t\right)  \right)  ^{-1}.
\]
Then%
\begin{align*}
\underset{n\geq k}{\sum}\left(  \underset{j=k}{\overset{n}{\sum}}\left[
\mathcal{T}^{-1}\right]  _{k,j}\left[  \mathcal{T}\right]  _{j,n}\right)
\frac{t^{n}}{n!}  &  =\underset{j\geq k}{\sum}\left[  \mathcal{T}^{-1}\right]
_{k,j}\left(  \underset{n\geq j}{\sum}\left[  \mathcal{T}\right]  _{j,n}%
\frac{t^{n}}{n!}\right) \\
&  =\left(  g\circ \omega \right)  \left(  t\right)  \underset{j\geq k}{\sum
}\left[  \mathcal{T}^{-1}\right]  _{k,j}\frac{\left(  \varphi \circ
\omega \left(  t\right)  \right)  ^{j}}{j!}\\
&  =g\left(  \omega \left(  t\right)  \right)  \frac{1}{k!}\frac{\left(
\frac{\psi}{h\circ \psi}\circ \varphi \circ \omega \left(  t\right)  \right)  ^{k}%
}{g\left(  \psi \circ \varphi \circ \omega \left(  t\right)  \right)  }\\
&  =\frac{t^{k}}{k!},
\end{align*}
which gives $\underset{j=k}{\overset{n}{\sum}}\left[  \mathcal{T}^{-1}\right]
_{k,j}\left[  \mathcal{T}\right]  _{j,n}=\left[  \mathcal{I}\right]  _{k,n}.$
By setting $\left(  \mathcal{T}^{-1}\right)  ^{0}=\mathcal{I},$ it is obvious
that the theorem is true. Otherwise, for $s\geq1,$ by Lemma \ref{L0}, it
suffices to prove the proposition for $a_{j}=1,$ $j\geq1.$ Since $\mathcal{T}$
is an upper triangular matrix, then%
\[
\left[  \mathcal{T}^{-s-1}\right]  _{k,n}=\left[  \mathcal{T}^{-s}%
\mathcal{T}^{-1}\right]  _{k,n}=\underset{j=k}{\overset{n}{\sum}}\left[
\mathcal{T}^{-s}\right]  _{k,j}\left[  \mathcal{T}^{-1}\right]  _{j,n}.
\]
So, if we set $F_{k,s}\left(  t\right)  =\underset{n\geq k}{\sum}\left[
\mathcal{T}^{-s}\right]  _{k,n}\frac{t^{n}}{n!},$ $k\geq1,$ we get%
\begin{align*}
F_{k,s}\left(  t\right)   &  =\underset{n\geq k}{\sum}\left(  \underset
{j=k}{\overset{n}{\sum}}\left[  \mathcal{T}^{-\left(  s-1\right)  }\right]
_{k,j}\left[  \mathcal{T}^{-1}\right]  _{j,n}\right)  \frac{t^{n}}{n!}\\
&  =\underset{j\geq k}{\sum}\left[  \mathcal{T}^{-\left(  s-1\right)
}\right]  _{k,j}\left(  \underset{n\geq j}{\sum}\left[  \mathcal{T}%
^{-1}\right]  _{j,n}\frac{t^{n}}{n!}\right) \\
&  =\frac{1}{g\circ \psi \left(  t\right)  }\underset{j\geq k}{\sum}\left[
\mathcal{T}^{-\left(  s-1\right)  }\right]  _{k,j}\frac{1}{j!}\left(
\frac{\psi}{h\circ \psi}\left(  t\right)  \right)  ^{j}\\
&  =\frac{1}{g\circ \psi \left(  t\right)  }F_{k,s-1}\left(  \frac{\psi}%
{h\circ \psi}\left(  t\right)  \right)
\end{align*}
and since $F_{k,0}\left(  t\right)  =\frac{t^{k}}{k!},$ then, recursively on
$s,$ this last relation completes the proof.
\end{proof}

\begin{remark}
We note here that for $s=1$ in Theorem \ref{T1} there is no contradiction
because Lagrange inversion formula proves the identity%
\begin{equation}
\left(  \frac{d}{dt}\right)  _{\! \!t=0}^{\! \!n-1}\left(  \left(  h\left(
t\right)  \right)  ^{n}\frac{d}{dt}\left(  \left(  \varphi \left(  t\right)
\right)  ^{k}g\left(  t\right)  \right)  \right)  =\left(  \frac{d}%
{dt}\right)  _{\! \!t=0}^{\! \!n}\left(  \left(  \varphi \circ \omega \left(
t\right)  \right)  ^{k}g\left(  \omega \left(  t\right)  \right)  \right)
.\label{4}%
\end{equation}

\end{remark}

\noindent The following corollaries represent some particular cases of
Theorems \ref{T1} and \ref{T2} .

\noindent For $h\left(  t\right)  =1$ in Theorems \ref{T1} and \ref{T2} we
obtain the following corollary.

\begin{corollary}
\label{C1}Let $\left(  a_{n};n\geq1\right)  ,$ be a sequence of complex
numbers with $a_{j}\neq0,$ $j\geq1,$ $\varphi,g$ be power series such that
$\varphi \left(  0\right)  =0$ and let $\mathcal{T}$ be the matrix whose its
$\left(  k,n\right)  $-th entry is%
\begin{equation}
\left[  \mathcal{T}\right]  _{k,n}=\frac{a_{k}}{a_{n}}\frac{1}{k!}\left(
\frac{d}{dt}\right)  _{\! \!t=0}^{\! \!n}\left(  \left(  \varphi \left(
t\right)  \right)  ^{k}g\left(  t\right)  \right)  . \label{24}%
\end{equation}
Then%
\begin{equation}
\left[  \mathcal{T}^{s}\right]  _{k,n}=\frac{a_{k}}{a_{n}}\frac{1}{k!}\left(
\frac{d}{dt}\right)  _{\! \!t=0}^{\! \!n}\left(  \left(  \varphi^{\left \langle
s\right \rangle }\left(  t\right)  \right)  ^{k}\underset{i=0}{\overset
{s-1}{\prod}}g\left(  \varphi^{\left \langle i\right \rangle }\left(  t\right)
\right)  \right)  ,\text{ \ }s=0,1,2,\ldots, \label{8}%
\end{equation}
and%
\begin{equation}
\left[  \mathcal{T}^{-s}\right]  _{k,n}=\frac{a_{k}}{a_{n}}\frac{1}{k!}\left(
\frac{d}{dt}\right)  _{\! \!t=0}^{\! \!n}\left(  \frac{\left(  \psi
^{\left \langle s\right \rangle }\left(  t\right)  \right)  ^{k}}{g\left(
\psi \left(  t\right)  \right)  \cdots g\left(  \psi^{\left \langle
s\right \rangle }\left(  t\right)  \right)  }\right)  ,\text{ \ }%
s=0,1,2,\ldots. \label{9}%
\end{equation}

\end{corollary}

\noindent A particular case of Theorems \ref{T1} and \ref{T2} when
$\psi:\equiv \varphi^{\left \langle -1\right \rangle }\equiv \omega,$ is given by
the following corollary.

\begin{corollary}
\label{C5}Let $\mathcal{T}$ be the matrix whose its $\left(  k,n\right)  $-th
entry is given by (\ref{1}) for which $\varphi^{\left \langle -1\right \rangle
}\left(  t\right)  :=\psi \left(  t\right)  =th\left(  \psi \left(  t\right)
\right)  .$Then, there holds%
\begin{equation}
\left[  \mathcal{T}^{s}\right]  _{k,n}=\frac{a_{k}}{a_{n}}\dbinom{n}{k}\left(
\frac{d}{dt}\right)  _{\! \!t=0}^{\! \!n-k}\left(  g\left(  \psi \left(
t\right)  \right)  \right)  ^{s},\  \ s\in \mathbb{Z}. \label{7}%
\end{equation}

\end{corollary}

\noindent For $h\left(  t\right)  =1$ and $\varphi \left(  t\right)  =t$ in
Theorems \ref{T1} and \ref{T2} we obtain the following corollary.

\begin{corollary}
\label{C2}Let $\left(  g_{n};n\geq0\right)  $ and $\left(  a_{n}%
;n\geq1\right)  ,$ be sequences of complex numbers with $a_{j}\neq0,$
$j\geq1,$ and let $\mathcal{T}$ be the matrix whose its $\left(  k,n\right)
$-th entry is%
\begin{equation}
\left[  \mathcal{T}\right]  _{k,n}=\frac{a_{k}}{a_{n}}g_{n-k}\delta_{\left(
n\geq k\right)  }. \label{10}%
\end{equation}
Then, by setting $g\left(  t\right)  =\underset{j\geq0}{\sum}g_{j}t^{j}$ and
$\mathcal{T}^{0}:=\mathcal{I}$ be the identity matrix, we get%
\begin{equation}
\left[  \mathcal{T}^{s}\right]  _{k,n}=\frac{a_{k}}{a_{n}}\frac{1}{\left(
n-k\right)  !}\left(  \frac{d}{dt}\right)  _{\! \!t=0}^{\! \!n-k}\left(
g\left(  t\right)  \right)  ^{s},\  \  \ s\in \mathbb{Z}. \label{11}%
\end{equation}

\end{corollary}

\noindent For $g\left(  t\right)  =h\left(  t\right)  =1,$ Theorems \ref{T1}
and \ref{T2} can be expressed in terms of partial Bell polynomials as it is
showen in the following corollary.

\begin{corollary}
\label{C3}Let $\left(  a_{n};n\geq1\right)  ,$ be a sequence of complex
numbers with $a_{j}\neq0,$ $j\geq1,$ $\varphi$ be a power series such that
$\varphi \left(  0\right)  =0,$ $\varphi^{\prime}\left(  0\right)  \neq0$ and
let $\varphi^{\left \langle -1\right \rangle }:=\psi$ be its compositional
inverse. Then, for the matrix $\mathcal{T}$ defined by%
\begin{equation}
\left[  \mathcal{T}\right]  _{k,n}=\frac{a_{k}}{a_{n}}B_{n,k}\left(
\varphi \right)  \label{14}%
\end{equation}
we have%
\begin{equation}
\left[  \mathcal{T}^{s}\right]  _{k,n}=\frac{a_{k}}{a_{n}}B_{n,k}\left(
\varphi^{\left \langle s\right \rangle }\right)  ,\text{ \ }s\in \mathbb{Z}.
\label{15}%
\end{equation}

\end{corollary}

\noindent Engbers et al. \cite{eng}, gave interesting combinatorial
interpretations of some cases of $B_{n,k}\left(  \psi \right)  $ when
$\varphi \left(  t\right)  =\underset{n\in R}{\sum}\varphi_{n}\frac{t^{n}}%
{n!},$ with $R$ is a subset of the set $\left \{  1,2,3,\ldots \right \}  $
containing $1.$ Also, when $\varphi$ and $g$ have integral coefficients,
Mihoubi and Rahmani \cite{mih3} gave combinatorial interpretations with
counting colored partitions on a finite set.

\noindent For $g\left(  t\right)  =\left(  \frac{t}{\varphi \left(  t\right)
}\right)  ^{\alpha}$ and $h\left(  t\right)  =1,$ Theorems \ref{T1} and
\ref{T2} become as follows.

\begin{corollary}
\label{C4}Let $\left(  a_{n};n\geq1\right)  ,$ be a sequence of complex
numbers with $a_{j}\neq0,$ $j\geq1,$ $\varphi$ be a power series such that
$\varphi \left(  0\right)  =0,$ $\varphi^{\prime}\left(  0\right)  \neq0$ and
let $\mathcal{T}$ be the matrix whose its $\left(  k,n\right)  $-th entry is%
\begin{equation}
\left[  \mathcal{T}\right]  _{k,n}=\frac{a_{k}}{a_{n}}\dbinom{n}{k}\left(
\frac{d}{dt}\right)  _{\! \!t=0}^{\! \!n-k}\left(  \left(  \frac{t}%
{\varphi \left(  t\right)  }\right)  ^{\alpha-k}\right)  . \label{16}%
\end{equation}
Then%
\begin{equation}
\left[  \mathcal{T}^{s}\right]  _{k,n}=\frac{a_{k}}{a_{n}}\dbinom{n}{k}\left(
\frac{d}{dt}\right)  _{\! \!t=0}^{\! \!n-k}\left(  \left(  \frac{t}%
{\varphi^{\left \langle s\right \rangle }\left(  t\right)  }\right)  ^{\left(
\alpha-k\right)  \operatorname{sign}s}\right)  ,\text{ \ }s\in \mathbb{Z}%
,\  \alpha \in \mathbb{R}, \label{17}%
\end{equation}
where $\mathbb{R}$ is the set of real numbers and $\operatorname{sign}%
s=\delta_{\left(  s\geq0\right)  }-\delta_{\left(  s\leq0\right)  }.$
\end{corollary}

\begin{example}
For $\varphi \left(  t\right)  =\frac{t}{1-\beta t}$ in Corollary \ref{C2} we
get $\varphi^{\left \langle s\right \rangle }\left(  t\right)  =\frac{t}{1-\beta
st}$ and%
\begin{equation}
\left[  \mathcal{T}^{s}\right]  _{k,n}=s^{n-k}\left[  \mathcal{T}\right]
_{k,n}=\dbinom{n}{k}\left(  \beta s\right)  ^{n-k}\left(  \alpha+n-1\right)
_{n-k},\text{ \ }s\in \mathbb{Z}. \label{26}%
\end{equation}
where $\left(  \alpha \right)  _{0}=1,\  \left(  \alpha \right)  _{n}%
=\alpha \left(  \alpha-1\right)  \cdots \left(  \alpha-n+1\right)  ,$ $n\geq1.$
\end{example}

\begin{example}
For $\varphi \left(  t\right)  =\left(  1+t\right)  ^{\alpha}-1$ and $g\left(
t\right)  =1$ in Corollary \ref{C4} we get $\varphi^{\left \langle
s\right \rangle }\left(  t\right)  =\left(  1+t\right)  ^{\alpha^{s}}-1$ and%
\begin{equation}
\left[  \mathcal{T}^{s}\right]  _{k,n}=\frac{1}{k!}\underset{j=0}{\overset
{k}{\sum}}\left(  -1\right)  ^{k-j}\dbinom{k}{j}\left(  \alpha^{s}j\right)
_{n-k},\  \  \alpha \neq0,\text{ }s\in \mathbb{Z}. \label{31}%
\end{equation}
One can verify that we have%
\[
\underset{n\geq0}{\sum}\left[  \mathcal{T}^{s}\right]  _{k,n+k}\frac{t^{n}%
}{n!}=\frac{1}{k!}\left(  \left(  1+t\right)  ^{\alpha^{s}}-1\right)  ^{k},
\]
from which it results, by derivation, that this sequence obeys to the
recurrence relation%
\[
\left[  \mathcal{T}^{s}\right]  _{k,n+k+1}=\alpha^{s}\left[  \mathcal{T}%
^{s}\right]  _{k-1,n+k-1}+\left(  \alpha^{s}k-n\right)  \left[  \mathcal{T}%
^{s}\right]  _{k,n+k}.
\]

\end{example}

\section{Powers of special non-triangular matrices}

Let $\left(  a_{j};j\geq1\right)  $ be a sequence of real numbers with
$a_{j}\neq0,$ $j\geq1,$ $\left(  b_{j};j\geq1\right)  $ be a sequence of
non-negative integers, $\varphi,h$ be power series such that $\varphi \left(
0\right)  =0,$ $\varphi^{\prime}\left(  0\right)  h\left(  0\right)  \neq0,$
and let $\mathcal{N}$ be the matrix whose its $\left(  k,n\right)  $-th entry
is%
\begin{equation}
\left[  \mathcal{N}\right]  _{k,n}=\frac{a_{k}}{a_{n}}\frac{1}{k!}\left(
\frac{d}{dt}\right)  _{\! \!t=0}^{\! \!n+b_{k}-1}\left(  \left(  h\left(
t\right)  \right)  ^{n+b_{k}}\frac{d}{dt}\left(  \left(  \varphi \left(
t\right)  \right)  ^{k}\right)  \right)  ,\  \ n\geq1,\ k\geq1. \label{23}%
\end{equation}
We see below that $\left[  \mathcal{N}\right]  _{k,n}$ must be%
\[
\left[  \mathcal{N}\right]  _{k,n}=\frac{a_{k}}{a_{n}}B_{n+b_{k},k}\left(
\varphi \circ \omega \right)  ,\  \ n\geq1,\ k\geq1,
\]
where $z=\omega \left(  t\right)  $ solution of the equation $z=th\left(
z\right)  .$\newline In this section, we give the $s$-th power of matrix whose
$\left(  k,n\right)  $-th entry is in form of (\ref{23}).

\begin{theorem}
[Closed form for $\mathcal{N}^{s},\ s\in \mathbb{N}$]\label{T3}Let
$\mathcal{N}$ be the matrix whose its $\left(  k,n\right)  $-th entry is given
by (\ref{23}). Then, for any non-negative integer $s,$ there holds%
\begin{equation}
\left[  \mathcal{N}^{s}\right]  _{k,n}=\frac{a_{k}}{a_{n}}B_{n+sb_{k}%
,k}\left(  \left(  \varphi \circ \omega \right)  ^{\left \langle s\right \rangle
}\right)  ,\  \ n\geq1,\ k\geq1. \label{32}%
\end{equation}

\end{theorem}

\begin{proof}
By setting $\mathcal{N}^{0}=\mathcal{I},$ it is obvious that the theorem is
true for $s=0.$ By Lemma \ref{L0}, it suffices to prove the proposition for
$a_{j}=1,$ $j\geq1.$ Indeed, if we set%
\[
F_{k,0}\left(  t\right)  =\frac{t^{k}}{k!},\  \ F_{k,s}\left(  t\right)
=\underset{n\geq1}{\sum}\left[  \mathcal{N}^{s}\right]  _{k,n}\frac{t^{n}}%
{n!},\  \ s\geq1,\  \ k\geq1,
\]
we get using Lagrange inversion formula%
\begin{align*}
F_{k,1}\left(  t\right)   &  =\frac{1}{k!}\underset{n\geq1}{\sum}\frac{t^{n}%
}{n!}\left(  \frac{d}{du}\right)  _{\! \!t=0}^{\! \!n+b_{k}-1}\left(  \left(
h\left(  u\right)  \right)  ^{n+b_{k}}\frac{d}{du}\left(  \left(
\varphi \left(  u\right)  \right)  ^{k}\right)  \right)  \\
&  =\frac{1}{k!}\left(  \frac{d}{dt}\right)  ^{b_{k}}\left[  \underset{n\geq
1}{\sum}\frac{t^{n}}{n!}\left(  \frac{d}{du}\right)  _{\! \!t=0}^{\! \!n-1}%
\left(  \left(  h\left(  u\right)  \right)  ^{n}\frac{d}{du}\left(  \left(
\varphi \left(  u\right)  \right)  ^{k}\right)  \right)  \right]  \\
&  =\frac{1}{k!}\left(  \frac{d}{dt}\right)  ^{b_{k}}\left(  \left(
\varphi \left(  z\right)  \right)  ^{k}\right)  .
\end{align*}
For $s\geq1,$ we have $\left[  \mathcal{N}^{s+1}\right]  _{k,n}=\left[
\mathcal{N}^{s}\mathcal{N}\right]  _{k,n}=\underset{j\geq1}{\sum}\left[
\mathcal{N}^{s}\right]  _{k,j}\left[  \mathcal{N}\right]  _{j,n}.$ So, we get%
\begin{align*}
F_{k,s}\left(  t\right)   &  =\underset{n\geq1}{\sum}\left(  \underset{j\geq
1}{\sum}\left[  \mathcal{N}^{s-1}\right]  _{k,j}\left[  \mathcal{N}\right]
_{j,n}\right)  \frac{t^{n}}{n!}\\
&  =\underset{j\geq1}{\sum}\left[  \mathcal{N}^{s-1}\right]  _{k,j}\left(
\underset{n\geq1}{\sum}\left[  \mathcal{N}\right]  _{j,n}\frac{t^{n}}%
{n!}\right)  \\
&  =\left(  \frac{d}{dt}\right)  ^{\! \!b_{k}}\underset{j\geq1}{\sum}\left[
\mathcal{N}^{s-1}\right]  _{k,j}\frac{\left(  \varphi \left(  z\right)
\right)  ^{j}}{j!}\\
&  =\left(  \frac{d}{dt}\right)  ^{\! \!b_{k}}F_{k,s-1}\left(  \varphi \left(
z\right)  \right)  .
\end{align*}
So, recursively on $s,$ we obtain%
\begin{align*}
F_{k,s}\left(  t\right)   &  =\left(  \frac{d}{dt}\right)  ^{\! \!sb_{k}%
}F_{k,0}\left(  \left(  \varphi \circ \omega \right)  ^{\left \langle
s\right \rangle }\left(  t\right)  \right)  \\
&  =\frac{1}{k!}\left(  \frac{d}{dt}\right)  ^{\! \!sb_{k}}\left(  \left(
\varphi \circ \omega \right)  ^{\left \langle s\right \rangle }\left(  t\right)
\right)  ^{k}\\
&  =\underset{n\geq1}{\sum}B_{n+sb_{k},k}\left(  \left(  \varphi \circ
\omega \right)  ^{\left \langle s\right \rangle }\right)  \frac{t^{n}}{n!}.
\end{align*}

\end{proof}

\noindent In particular, for $h\left(  t\right)  =1,$ Theorem \ref{T3} being
equivalent to:

\begin{corollary}
\label{C}Let $\mathcal{N}$ be the matrix whose its $\left(  k,n\right)  $-th
entry is given by%
\[
\left[  \mathcal{N}\right]  _{k,n}=\frac{a_{k}}{a_{n}}B_{n+b_{k},k}\left(
\varphi \right)  .
\]
Then, for any non-negative integer $s,$ there holds%
\begin{equation}
\left[  \mathcal{N}^{s}\right]  _{k,n}=\frac{a_{k}}{a_{n}}B_{n+sb_{k}%
,k}\left(  \varphi^{\left \langle s\right \rangle }\right)  . \label{40}%
\end{equation}

\end{corollary}

\begin{remark}
By the equality $\mathcal{N}^{s+t}=\mathcal{N}^{s}\mathcal{N}^{t},$ it results%
\begin{equation}
B_{n+\left(  s+t\right)  b_{k},k}\left(  \varphi^{\left \langle
s+t\right \rangle }\right)  =\underset{j\geq0}{\sum}B_{j+sb_{k},k}\left(
\varphi^{\left \langle s\right \rangle }\right)  B_{n+tb_{k},j}\left(
\varphi^{\left \langle t\right \rangle }\right)  . \label{34}%
\end{equation}

\end{remark}

\begin{example}
For $\varphi_{\alpha}\left(  t\right)  =\frac{t}{1-\alpha t}=\underset{j\geq
1}{\sum}\alpha^{j-1}t^{j},$ $\alpha \in \mathbb{R},$ we get%
\[
\varphi_{\alpha}^{\left \langle s\right \rangle }\left(  t\right)
=\varphi_{s\alpha}\left(  t\right)  \text{ and }B_{n,k}\left(  \varphi
_{\alpha}^{\left \langle s\right \rangle }\right)  =\left(  s\alpha \right)
^{n-k}%
%TCIMACRO{\QDATOPD{\lfloor}{\rfloor}{n}{k}}%
%BeginExpansion
\genfrac{\lfloor}{\rfloor}{0pt}{0}{n}{k}%
%EndExpansion
.
\]
Then, by Corollary \ref{C}, there holds%
\begin{equation}
\left[  \mathcal{N}^{s}\right]  _{k,n}=\frac{a_{k}}{a_{n}}\left(
s\alpha \right)  ^{n+sb_{k}-k}%
%TCIMACRO{\QDATOPD{\lfloor}{\rfloor}{n+sb_{k}}{k}}%
%BeginExpansion
\genfrac{\lfloor}{\rfloor}{0pt}{0}{n+sb_{k}}{k}%
%EndExpansion
.\label{35}%
\end{equation}
where $%
%TCIMACRO{\QDATOPD{\lfloor}{\rfloor}{n}{k}}%
%BeginExpansion
\genfrac{\lfloor}{\rfloor}{0pt}{0}{n}{k}%
%EndExpansion
$ is the $\left(  n,k\right)  $-th Lah number.
\end{example}

\begin{example}
For $\varphi_{\alpha}\left(  t\right)  =\left(  1+t\right)  ^{\alpha
}-1=\underset{j\geq1}{\sum}\dbinom{\alpha}{j}t^{j},$ $\alpha \in \mathbb{R}%
-\left \{  0\right \}  ,$ we get%
\[
\varphi_{\alpha}^{\left \langle s\right \rangle }\left(  t\right)
=\varphi_{\alpha^{s}}\left(  t\right)  \text{ and }B_{n,k}\left(
\varphi_{\alpha}^{\left \langle s\right \rangle }\right)  =\frac{1}{k!}%
\underset{j=0}{\overset{k}{\sum}}\left(  -1\right)  ^{k-j}\dbinom{k}{j}\left(
\alpha^{s}j\right)  _{n-k}.
\]
Then, by Corollary \ref{C}, there holds%
\begin{equation}
\left[  \mathcal{N}^{s}\right]  _{k,n}=\frac{a_{k}}{a_{n}}\frac{1}%
{k!}\underset{j=0}{\overset{k}{\sum}}\left(  -1\right)  ^{k-j}\dbinom{k}%
{j}\left(  \alpha^{s}j\right)  _{n+sb_{k}-k},\  \  \alpha \neq0.\label{38}%
\end{equation}

\end{example}

\begin{example}
For $\varphi_{\alpha,m}\left(  t\right)  =\frac{t}{\left(  1-\alpha
t^{m}\right)  ^{\frac{1}{m}}}=\underset{j\geq0}{\sum}\dbinom{-\frac{1}{m}}%
{j}\alpha^{j}t^{mj+1},$ $\alpha \in \mathbb{R},$ $m\in \mathbb{N-}\left \{
0\right \}  ,$ we get $\varphi_{\alpha,m}^{\left \langle s\right \rangle }\left(
t\right)  =\varphi_{s\alpha,m}\left(  t\right)  $ and%
\[
B_{n,k}\left(  \varphi_{\alpha,m}^{\left \langle s\right \rangle }\right)
=\left \{
\begin{array}
[c]{c}%
\frac{n!}{k!}\dbinom{-\frac{k}{m}}{\left(  n+k\right)  /m}\left(
s\alpha \right)  ^{\frac{n+k}{m}}\text{ if }m\mid n+k,\\
\smallskip \\
0\text{ \  \  \  \  \  \  \  \  \  \  \  \  \  \  \  \  \  \  \ otherwise.}%
\end{array}
\right.
\]
Then, by Corollary \ref{C}, there holds%
\[
\left[  \mathcal{N}^{s}\right]  _{k,n}=\left \{
\begin{array}
[c]{c}%
\frac{a_{k}}{a_{n}}\frac{\left(  n+sb_{k}\right)  !}{k!}\dbinom{-\frac{k}{m}%
}{\left(  n+k+sb_{k}\right)  /m}\left(  s\alpha \right)  ^{\frac{n+k+sb_{k}}%
{m}}\text{ if }m\mid n+k+sb_{k},\\
\smallskip \\
0\text{ \  \  \  \  \  \  \  \  \  \  \  \  \  \  \  \  \  \  \ otherwise.}%
\end{array}
\right.
\]

\end{example}

\end{document}